\documentclass[11pt, reqno]{article}

\title{Cutoff in the Bernoulli-Laplace model with $O(n)$ swaps } 
\author{
Joseph S. Alameda\thanks{(alameda@usna.edu).} \and
Caroline Bang\thanks{(cbang@iastate.edu).} \and
Zachary Brennan\thanks{(brennanz@iastate.edu).} \and 
David P. Herzog\thanks{(dherzog@iastate.edu).}\and 
J{\"u}rgen Kritschgau\thanks{(jkritsch@andrew.cmu.edu).} \and
Elizabeth Sprangel\thanks{(sprangel@iastate.edu).}
}

\usepackage{amsfonts, amssymb, amsmath, amsthm,enumerate,esint,multicol, mathptmx, tikz}
\usepackage[colorlinks=true, pdfstartview=FitV, linkcolor=blue,
            citecolor=blue, urlcolor=blue]{hyperref}
\definecolor{ourRed}{RGB}{255, 81, 81}
\usepackage[shortlabels]{enumitem}
\usepackage{amsmath}
\usepackage{accents}
\usepackage{comment}
\usepackage{graphicx}
\usepackage[capitalize,nameinlink,noabbrev]{cleveref}
\usepackage{physics} 
\usepackage{setspace}
\usepackage{geometry}

\numberwithin{equation}{section}

\newtheorem{Theorem}{Theorem}[section]
\newtheorem{Proposition}[Theorem]{Proposition}
\newtheorem{Lemma}[Theorem]{Lemma}

\theoremstyle{definition}
\newtheorem{Definition}{Definition}[section]

\newtheorem{Assumption}{Assumption}
\newtheorem{Coupling}{Coupling}

\makeatother\newtheorem{Remark}[Theorem]{Remark}

\newcommand{\PP}{\mathbf{P}}

\newcommand{\N}{\mathbf{N}}
\newcommand{\Z}{\mathbf{Z}}
\newcommand{\cals}{\mathcal{S}}

\newcommand{\caly}{\mathcal{Y}}

\newcommand{\E}{\mathbf{E}}

\newcommand{\RR}{\mathbf{R}}

\newcommand{\calx}{\mathcal X}

\newcommand{\tmix}{t_{\text{mix}}}

\newcommand{\TV}[1]{||#1||_{TV}}

\begin{document}\maketitle

\begin{abstract}
This paper considers the $(n,k)$-Bernoulli--Laplace model in the case when there are two urns, the total number of red and white balls is the same, and the number of selections $k$ at each step is on the same asymptotic order as the number of balls $n$ in each urn.  Our main focus is on the large-time behavior of the corresponding Markov chain tracking the number of red balls in a given urn.  Under reasonable assumptions on the asymptotic behavior of the ratio $k/n$ as $n\rightarrow \infty$, cutoff in the total variation distance is established.  A cutoff window is also provided.  These results, in particular, partially resolve an open problem posed by Eskenazis and Nestoridi in~\cite{EN}.     
\end{abstract}


\section{Introduction}

Throughout this paper, we study the \emph{$(n,k)$-Bernoulli-Laplace model}.  In the model, there are two urns, a \emph{left} urn and a \emph{right} urn, each of which contains exactly $n$ balls.  Of the total $2n$ balls contained in both urns, $n$ are colored red and $n$ are colored white.  Starting from a given coloration of $n$ balls in each urn, at each step $k$ balls are selected uniformly at random without replacement from each urn.  The selected balls are then swapped and placed in the opposite urn.  The process then repeats itself.  Letting $X_t$ denote the number of red balls in the left urn after $t$ swaps, the process $(X_t)$ is Markov.  Our main goal is to understand how long it takes for the chain to be within $\epsilon>0$ of its stationary distribution $\pi$ in total variation.  Our focus in this paper is on the case when the number of swaps $k$ is of order $n$ where $n\gg1$.  The main result of this paper partially resolves an open question posed by Eskenazis and Nestoridi in~\cite{EN}.

Our interest in the $(n,k)$-Bernoulli-Laplace model comes from shuffling large decks of cards.  Mapping the above model to this setting, the deck of cards has size $2n\gg1$ and at each step of the shuffle we cut the deck into two equal piles of $n$ cards, shuffle each deck independently and perfectly, reassemble the deck and then move the top $k$ cards to the bottom.  This process repeats itself until sufficient mixing is achieved.  From this description, it follows that the $(n,k)$-Bernoulli-Laplace model describes this card shuffling algorithm \emph{without} the separate step of shuffling each of the smaller decks independently and perfectly at each step.  See~\cite{NW_19} for further details.  

\subsection{Preliminaries}  Before discussing existing results in the literature and the results of this paper, we first fix some notation and terminology.  

Throughout, $\calx= \{0,1,2,\ldots, n\}$ denotes the state space of the $(n,k)$-Bernoulli-Laplace chain $(X_t)$, and $p_t(x,y)$ denotes the associated probability of transitioning from $x\in\calx$ to $y\in \calx$ in $t$ steps; that is,
\begin{align}
\label{eqn:trans}
p_t(x,y):= \PP_x \{ X_t =y \}
\end{align}  
where the subscript in the probability $\PP$ indicates that $X_0=x$.  One can write down the specific formulas for the transitions $p_t(x,y)$ (see~\cite{EN}) but these formulas are not particularly important in our analysis. Similar to $\PP_{X_0}\{ \cdots\}$, we use an analogous notation for the expectation $\E$; that is, $\E_{X_0}\{ \cdots \}$ indicates that $(X_t)$ has initial distribution $X_0$.  For any $A\subset \calx$ and $x\in \calx$, we let 
\begin{align}
P_t(x, A) := \sum_{y\in A} p_t(x,y) = \PP_x\{ X_t \in A\}
\end{align}
denote the probability of transitioning from state $x$ to the set $A$ in $t$ steps.  It is known (see~\cite{Tai_96}) that $(X_t)$ has a unique stationary distribution $\pi$ which is hypergeometric; specifically, $\pi$ satisfies
\begin{align}
\label{eqn:pidist}
\pi(\{x\})= \frac{\binom{n}{x} \binom{n}{n-x}}{\binom{2n}{n}}, \qquad x\in \calx. 
\end{align}  
Observe that each of the quantities above depends on the parameter $n\in \N$.  Throughout, unless otherwise specified (see, for example, the paragraph below), we will suppress this dependence.    

Let 
\begin{align}
d^{(n)}(t):= \max_{x\in \calx} \| P_t(x, \, \cdot \,) - \pi (\, \cdot \,) \|_{TV} = \tfrac{1}{2}\max_{x\in \calx} \sum_{y\in \calx} | p_t(x,y) - \pi(y)|,  
\end{align} 
and define for $\epsilon >0$ the \emph{mixing time} $\tmix^{(n)} (\epsilon)$ by 
\begin{align}
\tmix^{(n)}(\epsilon) = \min \{ t\in \N \, : \, d^{(n)}(t) \leq \epsilon \}. 
\end{align}
We say that the Markov chain $(X_t)$ exhibits \emph{cutoff} if  
\begin{align}
\label{eqn:cutoff}
\lim_{n\rightarrow \infty} \frac{\tmix^{(n)}(\epsilon)}{\tmix^{(n)}(1-\epsilon)} =1 \,\,\, \,\,\text{ for all }\, \,\,\, \, \epsilon \in (0, 1).  
\end{align}
If the Markov chain $(X_t)$ exhibits cutoff and for every $\epsilon \in (0,1)$ there exists a constant $c_\epsilon$ and a sequence $w_n$ satisfying   
\begin{align}
w_n=o(\tmix^{(n)}(1/2)) \qquad \text{ and } \qquad \tmix^{(n)}(\epsilon)- \tmix^{(n)}(1-\epsilon) \leq c_\epsilon w_n   \quad \text{ for all }\quad  n, 
\end{align}  
we say that $(X_t)$ has \emph{cutoff window} $w_n$.  For other preliminaries concerning mixing times of Markov chains, see~\cite{LPW_17}.

\subsection{Previous results and statement of the main result}
Existing results on mixing  times for the $(n,k)$-Bernoulli-Laplace model largely focus on the case when the number of swaps $k$ is \emph{much smaller} than the number of balls $n$ in each urn.  The earliest works of Diaconis and Shahshahani~\cite{DS_81, DS_87}  and Donnely, Floyd and Sudbury~\cite{DLS_94} treat the case when $k=1$ and establish cutoff in total variation and in the \emph{separation distance}, respectively.  Diaconis and Shahshahani proved their results by analyzing random walks on Cayley graphs of the symmetric group where edges correspond to transpostions. These results were extended to random walks on distance regular graphs in~\cite{B_98}.  See also~\cite{scarabotti1997time} for the Bernoulli-Laplace model with multiple urns in the case when $k=1$.  We refer to~\cite{school_02} which contains a signed generalization of the model.  

The case when the number of swaps $k>1$ in the Bernoulli-Laplace two-urn model was first studied by Nestoridi and White~\cite{NW_19}, where a number of estimates (not all sharp) were deduced for $\tmix(\epsilon)$ for general $k$.  These estimates were made sharp in the case when the number of swaps $k$ satisfies $k=o(n)$ in a joint paper of Eskenazis and Nestoridi~\cite{EN}, ultimately yielding the bound 
\begin{align}
\label{eqn:upperbksmall}
\frac{n}{4k}\log n - \frac{c(\epsilon) n}{k}\leq \tmix(\epsilon) \leq \frac{n}{4k} \log n + \frac{3n}{k}\log \log n + O\Big( \frac{n}{\epsilon^4 k}\Big).   
\end{align}
Thus, the model with $k=o(n)$ exhibits cutoff with cutoff window $\tfrac{n}{k} \log \log n$.  Estimates deduced in the case when $k=O(n)$ were not optimal; that is, in~\cite{NW_19} it was shown that if $k/n \rightarrow \lambda \in (0,1/2)$, then 
 \begin{align}
 \label{eqn:lowb1}
 \frac{\log n }{2 |\log (1-2\lambda)|} -c(\epsilon) \leq \tmix(\epsilon) \leq \frac{\log(n/\epsilon)}{2\lambda(1-\lambda)}.  
 \end{align} 
Moreover, it was conjectured in~\cite{EN} that the lower bound in~\eqref{eqn:lowb1} is sharp.  In general, it was left there as an open problem to determine the mixing time of the $(n,k)$-Bernoulli-Laplace model when $k/n\rightarrow \lambda \in (0,1/2)$.   In this paper, we make progress on solving this problem under \emph{reasonable} assumptions on the convergence rate $k/n\rightarrow \lambda \in (0,1/2)$, which we now describe.

%
\begin{Assumption}
\label{assump:1}
The parameter $k=k(n)$ in the $(n,k)$-Bernoulli-Laplace model satisfies the following conditions:
\begin{itemize}
\item[\textbf{(c0)}] $k/n\rightarrow \lambda \in (0, 1/2)$ as $n\rightarrow \infty.$  
\item[\textbf{(c1)}] There exists $\delta \in (0,1/2)$ for which $k/n\in (0, \delta)$ for all $n$.
\item[\textbf{(c2)}] $\Delta_n :=\tfrac{k}{n}- \lambda$ satisfies the asymptotic condition
\begin{align}
\Delta_n = o(1/\log n) \,\,\,\, \text{ as } \,\,\,\, n\rightarrow \infty. 
\end{align}
\end{itemize}

\begin{Remark}
\label{rem:1}
Although \cref{assump:1} is not explicitly employed in the paper~\cite{NW_19}, to the best of our knowledge it seems that one needs to impose some condition like the one above to make the asymptotic formulas previously used in~\cite{NW_19} to deduce a lower bound on $t_{\text{mix}}^{(n)}(\epsilon)$.  For further information on this point, we refer the reader to~\cref{rem:2}.   
\end{Remark}

\end{Assumption}

Throughout, using the notation in \cref{assump:1}, we define 
\begin{align}
\label{def:times}
s_n = \lambda^{-1} \log \log n \qquad \text{ and } \qquad t_n = \frac{\log n}{2| \log (1-2\lambda)|}.
\end{align}
Our main result is the following:
\begin{Theorem}
\label{thm:main}
Suppose that in the $(n,k)$-Bernoulli-Laplace chain, the parameter $k$ satisfies \cref{assump:1}.  
 For any $\epsilon \in (0,1)$, there exists a constants $c(\epsilon), N(\epsilon)>0$ depending only on $\epsilon$ such that 
\begin{align}
\label{eqn:upperb}
t_n - c(\epsilon) \leq \tmix^{(n)}(\epsilon) \leq t_n + 3s_n + 1 
\end{align}
for all $n\geq N(\epsilon)$. 
In particular, the $(n,k)$-Bernoulli-Laplace chain with parameter $k$ satisfying \cref{assump:1} has mixing time $t_n$ with cutoff window $s_n.$
\end{Theorem}
%

\section{Proof outline }
\label{sec:outline}

\subsection{The lower bound}
The proof of the lower bound with $k$ satisfying \cref{assump:1} was first estabished in \cite{NW_19}.  
We present the lower bound argument and fill in some details in \cref{sec:lowerb} for completeness. 
The lower bound argument follows the standard idea of finding a set $A\subseteq \calx$ where $|P_t (0,A)-\pi_n(A)|$ is large. In particular, the argument uses the estimate 
\[\TV{P_t(0, \, \cdot\,)-\pi_n(\, \cdot \,)} \geq \sup_{A\subseteq \mathcal X} |P_t(0,A)-\pi_n(A)|.\]
It is worth noting that the lower bound argument uses Chebyshev's inequality to bound the $P_t(0,A)$ from above and to bound $\pi_n(A)$ from below, for a clever choice of $A$. 
This argument is instructive and provides intuitive support for the lower bound being sharp since the application of Chebychev's inequality fails to bound $P_t(0,A)$ and $\pi_n(A)$ away from each other for $t\geq t_n$.
Our main contribution in this paper is the upper bound~\eqref{eqn:upperb}. 

\subsection{The upper bound~\eqref{eqn:upperb}}

While the broad structure of the argument giving the upper bound~\eqref{eqn:upperb} is similar to that used to obtain the upper bound in~\eqref{eqn:upperbksmall} in the case when $k=o(n)$~\cite{EN}, the details giving these results are quite different.  This is especially true because the asymptotics are more delicate when $k$ is order $n$ as in~\cref{assump:1}.  Furthermore, we cannot rely on a known result of Diaconis and Freedman~\cite{DF_1980} in the last step of the proof.  This was used crucially in~\cite{EN} to show that once two copies of the chain are sufficiently close, in this case at distance $o(k)$, then in one step the processes are distance $o(1)$ apart with high probability.  This step relies inherently the binomial approximation of the hypergeometric distribution (see, for example,~\cite{R_01}), which is not valid in our regime of interest since $k$ is order $n$.  The last step, i.e. going from distance $o(k)$ to $o(1)$, is arguably the most involved because we manufacture a \emph{discrete normal} approximation (in total variation) of the hypergeometric.  This same approximation was also used in~\cite{LCM_07} to establish a Berry-Esseen theorem for the hypergeometric distribution in nonstandard parameter cases.  For more on Barry-Esseen theorems, we refer the reader to Feller's book~\cite{F_08}.  All of these steps are carried out in detail in \cref{sec:osmallrootn} and \cref{sec:upperbound}.

\subsubsection{The coupling}        
The proof of the upper bound~\eqref{eqn:upperb} employs coupling methods.  Specifically, recall that if $\mu$ and $\nu$ are probability measures defined on all subsets of $\calx$, then we can write
\begin{align}
\| \mu - \nu \|_{TV} = \inf \{ \PP\{ X\neq Y \} \, : \, (X,Y) \text{ is a coupling of } \mu \text{ and } \nu \}.  
\end{align}
Here we recall that $(X,Y)$ is a \emph{coupling} of the measures $\mu$ and $\nu$ if $X$ and $Y$ are random variables on a common probability space $(\Omega, \mathcal{F}, \PP)$ with $X\sim \mu$ and $Y\sim \nu$.  Thus in order to achieve the claimed upper bound~\eqref{eqn:upperb}, we will utilize a particular coupling of two copies of the chain to bound the total variation distance from above. 

Recall that in the $(n,k)$-Bernoulli-Laplace model, $X_t$ denotes the number of red balls in the left urn at time $t$.  Thus if $(X_t)$ and $(Y_t)$ are two copies of the $(n,k)$-Bernoulli-Laplace chain, then $X_t$ and $Y_t$ are the number of red balls in two separate left urns at time $t$.  The coupling we use can be described as follows (see also~\cite{EN, NW_19}):   
%

\begin{Coupling}\label{coupling}
Let $(X_t)$ and $(Y_t)$ be two copies of the $(n,k)$-Bernoulli-Laplace chain and fix a time $t\geq 0$.  Given the distributions $X_t$ and $Y_t$ at time $t$, we define a coupling of $X_{t+1}$ and $Y_{t+1}$.  First, label the balls in the two left urns at time $t$ separately from $1$ to $n$ so that each red ball has a smaller label than each white ball. 
Furthermore, in a similar way label all of the balls in the two right urns from $n+1$ to $2n$ so that each red ball has a smaller label than each white ball. 
Uniformly and independently select subsets $A\subseteq \{1,\dots,n\}$ and $B\subseteq \{n+1,\dots, 2n\}$ with $|A|=|B|=k$.  To obtain $X_{t+1}$ and $Y_{t+1}$,  
swap the balls indexed by the elements of $A$ in each left urn with the balls in the corresponding right urn with index belonging to $B$.  \end{Coupling}


\begin{figure}
    \centering
    \begin{tikzpicture}[thick, scale = 1]
    
    \foreach \x in {0, 3, 7, 10}{
        \foreach \y in {0, 2, 4, 6}{
        \draw (\x -1, \y) arc [ start angle = 180, end angle = 360, x radius =1, y radius = .25];
        }
        \foreach \y in {2,6}{
        \draw (\x + 1, \y) arc [ start angle = 0, end angle = 180, x radius =1, y radius = .25];}
        \foreach \y in {0,4}{
        \draw[dashed] (\x + 1, \y) arc [ start angle = 0, end angle = 180, x radius =1, y radius = .25];
        \draw (\x - 1, \y ) -- (\x - 1, \y + 2);
        \draw (\x + 1, \y ) -- (\x + 1, \y + 2);
        \foreach \z in {.2,.8, 1.4}{
            \draw[fill=white] (\x -.4, \y + \z) circle (.25);}
        \foreach \z in {.5,1.2}{
            \draw[fill=white] (\x +.4, \y + \z) circle (.25);}
            }
            }
        
    \foreach \z in {.2, .8, 1.4}{
        \draw[fill = ourRed] (-.4, 4 + \z) circle (.25);}
    \foreach \z in {1.4, .8}{
        \draw[fill = ourRed] (2.6, 4 + \z) circle (.25);}
    \foreach \x in {0,7}{
        \draw[fill = ourRed] (\x - .4, 1.4) circle (.25);}
    \draw[fill = ourRed] (6.6, 5.4) circle (.25);
    \foreach \x in {3,10}{
        \foreach \z in {.2, .8, 1.4}{
            \draw[fill = ourRed] (\x - .4, \z) circle (.25);}
        \draw[fill = ourRed] (\x + .4, 1.2) circle (.25);}
    \foreach \z in {.2, .8, 1.4}{
        \draw[fill = ourRed] (9.6, 4 + \z) circle (.25);}
    \draw[fill = ourRed] (10.4, 5.2) circle (.25);

    \foreach \y in {0,4}{
        \foreach \z/\i in {1.4/1, .8/2, .2/3}{
            \draw ( - .4, \y + \z) node {\i};}
        \foreach \z/\i in {1.2/4, .5/5}{
            \draw (.4, \y + \z) node {\i};}
            }
    \foreach \y in {0,4}{
            \foreach \z/\i in {1.4/6, .8/7, .2/8}{
                \draw ( 2.6, \y + \z) node {\i};}
            \foreach \z/\i in {1.2/9, .5/10}{
                \draw (3.4, \y + \z) node {\i};}
                }
            
    \draw node [below] at (1.5, 3.5) {$X_t$};
    \draw node [below] at (1.5, -.5) {$Y_t$};
    \draw node [below] at (8.5, 3.5) {$X_{t+1}$};
    \draw node [below] at (8.5, -.5) {$Y_{t+1}$};
    
    \foreach \y in {0,4}{
        \foreach \x in {0, 7}{
            \draw (\x, \y - .5) node(\x\y){$L$};
        }
        \foreach \x in {3, 10}{
            \draw (\x, \y - .5) node(\x\y){$R$};
        }}
    
    \end{tikzpicture}
    \caption{An illustration of \cref{coupling} with $n=5$, $k=2$ and $A=\{1,3\}$ and $B=\{9,10\}$.}
    \label{fig:1}
\end{figure}
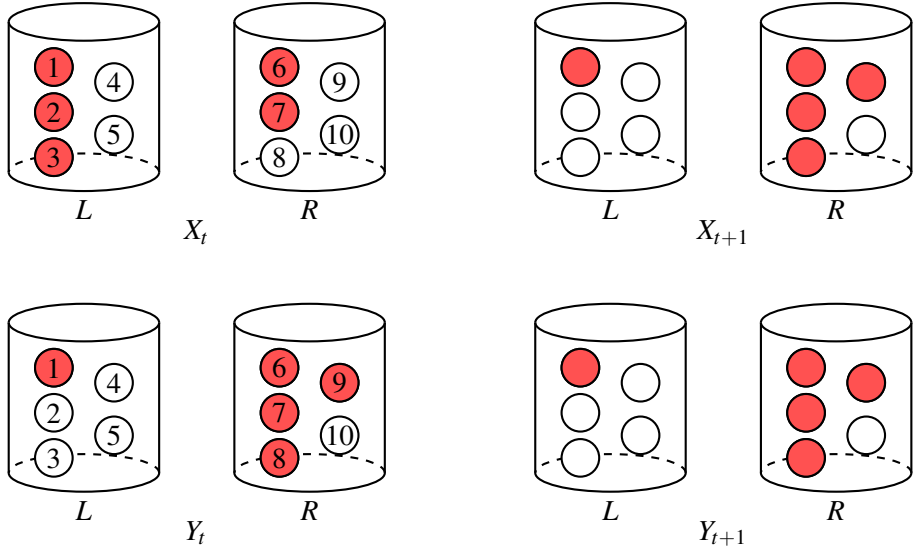


 Refer to \cref{fig:1} for an illustration of this coupling.  Importantly, two chains coupled as in \cref{coupling} have the property that $|X_{t+1}-Y_{t+1}|\leq |X_t - Y_t|$~\cite{NW_19}. 
Furthermore, we have the next proposition (see~\cite[Proposition 2]{EN}) whose argument is a  modification of the \emph{path coupling theorem} of Bubley and Dyer~\cite{BD_97}.   

\begin{Proposition}\label{prop:couple}
Let $(X_t)$ and $(Y_t)$ be two instances of the $(n,k)$-Bernoulli-Laplace chain coupled at each time $t\geq 1$ as in \cref{coupling} starting from $X_0, Y_0 \in \calx$. For $r>0$, let 
\begin{align}
\label{eqn:tcouple}
\tau_{\emph{couple}}(r) := \min\{t\,:\, |X_t-Y_t|\leq r\}.
\end{align}
Then, for every $t\in \mathbb{N}$, 
\[\PP(\tau_{\emph{couple}}(r)>t)\leq \Big( 1-\frac{2k(n-k)}{n^2}\Big)^t\frac{|X_0-Y_0|}{r}.\]

\end{Proposition}

\subsubsection{The steps in the upper bound argument}
In \cref{sec:osmallrootn} and~\cref{sec:upperbound}, we provide the details for the upper bound~\eqref{eqn:upperb}.  Here we lay out the argument giving the upper bound in several steps.  


\textbf{Step 1}.  Recalling $t_n$ as in~\eqref{def:times} and fixing $\kappa_1\gg 1$, we will  first prove that with high probability, there exists $t\leq t_n$ such that 
\begin{align*}
X_{t}, Y_{t} \in (\tfrac{n}{2}- \kappa_1 \sqrt{n}, \tfrac{n}{2}+ \kappa_1 \sqrt{n}). 
\end{align*} 
As a consequence, by $t_n$ steps, with high probability (depending on the choice of $\kappa_1$) the chains have been at distance order $\sqrt{n}$ apart.   This step will be achieved using the first and second eigenfunctions for the chain, which translates to precise first and second moment estimates for $(X_t)$.      

\textbf{Step 2}.  Recalling $s_n$ as in~\eqref{def:times} and fixing $\kappa_2 , \kappa_3\gg 1$, we will show that if $X_0, Y_0 \in (\tfrac{n}{2}-\kappa_2\sqrt{n},\tfrac{n}{2}+ \kappa_2 \sqrt{n})$, then with high probability (depending on $\kappa_2, \kappa_3$) if $X_t, Y_t$ are coupled according to \cref{coupling} there exists $t \leq s_n$ so that 
\begin{align*}
|X_t - Y_t| \leq \frac{\sqrt{n}}{\log \log n} \qquad \text{ and } \qquad X_t, Y_t \in (\tfrac{n}{2}- \kappa_3 \sqrt{n} (\log n)^{|P_\lambda|/2}, \tfrac{n}{2}+ \kappa_3 \sqrt{n} (\log n)^{|P_\lambda|/2})\end{align*}
where $P_\lambda =\log(1-2\lambda)  \tfrac{2}{\lambda}$.  Thus the chains become closer asymptotically but they may deviate slightly from $\tfrac{n}{2}$ by more than order $\sqrt{n}$.  Note that this deviation becomes particularly pronounced when $\lambda \approx 1/2$ by the definition of $P_\lambda$.     

\textbf{Step 3}.  In this step, we correct for the deviation in \textbf{Step 2} from $\tfrac{n}{2}$ in $2 s_n$ steps again making use of~\cref{coupling}.  That is, fix $\kappa_4 \gg 1$ and suppose $X_0, Y_0$ satisfy
\begin{align*}
|X_0 - Y_0| \leq \frac{\sqrt{n}}{\log \log n} \qquad \text{ and } \qquad X_t, Y_t \in (\tfrac{n}{2}- \kappa_3 \sqrt{n} (\log n)^{|P_\lambda|/2}, \tfrac{n}{2}+ \kappa_3 \sqrt{n} (\log n)^{|P_\lambda|/2}).\end{align*}
Supposing that the chains $(X_t)$ and $(Y_t)$ are coupled according to~\cref{coupling}, we will show that with high probability (depending on $\kappa_4 \gg 1$) there exists $t\leq 2 s_n$ so that 
\begin{align}
\label{eqn:step3}
|X_t - Y_t| \leq \frac{\sqrt{n}}{\log \log n} \qquad \text{ and } \qquad X_t, Y_t \in (\tfrac{n}{2}- \kappa_4 \sqrt{n} , \tfrac{n}{2}+ \kappa_4 \sqrt{n} ). \end{align}

\textbf{Step 4}.  At this point, by using the strong Markov property and combining \textbf{Step 1}, \textbf{Step 2} and \textbf{Step 3}, with high probability there exists $t\leq t_n + 3 s_n$ such that~\eqref{eqn:step3} is satisfied regardless of $X_0$, $Y_0$. From this point, we will see that in one additional step the processes are within distance $o(1)$ of each other with high probability.   This step is more involved than the other steps and constitutes its own section (cf. \cref{sec:upperbound}).  The main reason this step is more involved is that we have to understand the hypergeometric distributions that make up the one-step behavior of the chain within this  particular parameter range on the initial data.  Here, because $k$ is order $n$ instead  of $k=o(n)$, we cannot appeal to an existing result to do so.  Instead, we make use of some asymptotic estimates for the hypergeometric distributions  in~\cite{LCM_07} to compare the hypergeometrics with what we call \emph{discrete normal} distributions.  


%
%
%

\section{Lower bound}
\label{sec:lowerb}


Here we establish the lower bound on the mixing time~\eqref{eqn:upperb}.  First, however, we state some auxiliary results to be used throughout the paper.  For the reader, it may be more valuable to skip the proof of these results on a first pass and proceed directly to the proof of the lower bound.  The proof of the auxiliary results will be given later in~\cref{sec:aux}.

\subsection{Auxiliary lemmata}
For $x\in \calx$, define $f_i(x)$, $i=1,2,3$, by 
\begin{align}
\label{eqn:fs}
    f_1(x)= 1-\frac{2x}{n}, \qquad f_2(x) = 1-\frac{2(2n-1)x}{n^2}+\frac{2(2n-1)x(x-1)}{n^2(n-1)}, \,\,\, \text{ and } \,\,\, f_3(x)= 1- \frac{2x(n-x)}{n^2}.
\end{align}
We remark that while $f_1$ and $f_2$ are eigenfunctions for the Bernoulli-Laplace model with respective eigenvalues $f_1(k)$ and $f_2(k)$~\cite{EN}, $f_3(k)$ is contained in the righthand side of the bound in \cref{prop:couple}.  For further information on the dual Hahn eigenfunctions for the Bernoulli-Laplace chain, we refer the reader to~\cite{NW_19}.  See also~\cite{KZ_09} for further reading on polynomial eigenfunctions of Markov chains in general.     

One can check that $|f_1| \leq 1$ on $\calx$ and, after a short exercise optimizing the quadratic $f_2(x)$ on $\calx$, for $n\geq 2$ we have  
\begin{align}
\label{eqn:f2bound}
1- \frac{2n-1}{2n-2} \leq f_2(x) \leq 1  \,\,\,\,\, \text{ for all } \,\,\, x\in \calx.  
\end{align} 
We will also need the following basic structural facts about the functions $f_1$ and $f_2$.  
\begin{Lemma}
\label{lem:1}
For all $t\geq 0$ and all $X_0\in \calx$, we have the identities
\begin{align}
\label{id:1}
\E_{X_0} f_1(X_t) &=\Big(1-\frac{2k}{n}\Big)^t f_1(X_0)\\
\label{id:2} \E_{X_0} f_1(X_t)^2& = \frac{1}{2n-1}+ \frac{2n-2}{2n-1} f_2(k)^t f_2(X_0).  
\end{align}
Furthermore, let $q\geq0$ be a constant and suppose $X_0= \tfrac{n}{2}+ O( \sqrt{n} (\log n)^{q/2} )$ as $n\rightarrow \infty$.  Then we have the asymptotic formula 
\begin{align}
\label{eqn:f2as}
f_2(X_0) = 
O(n^{-1} (\log n)^q).  
\end{align}
\end{Lemma}

\begin{Lemma}
\label{lem:2}
Suppose that \cref{assump:1} is satisfied.  Let $t=t(n)\in \mathbf{N}$ satisfy $t\leq t_n$ for all $n$ where $t_n$ is as in~\eqref{def:times}.  Let $h_1, h_2$ be given by $f_2(k)=1-2h_1(k)$ and $(1-2\lambda)^2= 1- 2h_2(\lambda)$, and set $\Delta_n':= h_1(k)-h_2(\lambda)$ and $\Delta_n''= k(n-k)/n^2 - (\lambda- \lambda^2)$.  Then as $n\rightarrow \infty$
\begin{align}
\label{as:id1} f_1(k)^{t} &=(1-2\lambda)^{t}\bigg(1 -\frac{ 2t \Delta_n}{1-2\lambda} + O(t^2 \Delta_n^2) \bigg),\\
\label{as:id2}f_2(k)^t &= (1-2\lambda)^{2t}\bigg(1 - \frac{2t \Delta_n'}{1- 2\lambda} +O(t^2( \Delta_n')^2) \bigg),\\
\label{as:id3} f_3(k)^t&=  (1-2\lambda + 2\lambda^2)^t\bigg(1- \frac{2t \Delta_n''}{1- 2\lambda} +O(t^2( \Delta_n'')^2)   \bigg).
\end{align}
Furthermore, 
\begin{align}
\label{eqn:dtil}
|\Delta_n'| \leq 2 |\Delta_n| + O(1/n)   \qquad \text{ and } \qquad |\Delta_n''| \leq |\Delta_n| 
\end{align}
for all $n$. 
\end{Lemma}

\begin{Remark}
\label{rem:2}
Note that a similar formula to~\eqref{as:id1} is employed on~\cite[p.444]{NW_19} but without an assumption similar to~\cref{assump:1}.  Such a formula is not valid unless $k/n\rightarrow \lambda \in (0,1/2)$ sufficiently fast.     
\end{Remark}

\subsection{Proof of the lower bound}

If we let $f(x)=\sqrt{n-1}f_1(x)=\sqrt{n-1}(1-2x/n)$, then $f$ is tightly concentrated around its mean with respect to the stationary distribution $\pi_n$. 
This is true because $\pi_n$ has mean $n/2$ and the bulk of its mass is concentrated around its mean with exponentially decaying tails.  On the other hand, $\E_0 f(X_t) \approx (1-2\lambda)^{-c}$ for $t=t_n -c$ and large fixed $c$. 
Furthermore,  $P_t(0,\cdot)$ must concentrate values of $f(x)$ tightly around its mean, due to a constant bound on its variance. 
Thus, it is intuitive that $P_t(0,\cdot)$ must assign a significant portion of its mass away from $n/2$. 
In short, the bulk of the masses from $\pi_n$ and $P_t(0,\cdot)$ do not overlap, providing a lower bound on the total variation distance between these measures.  Note that this proof is presented in~\cite{NW_19}.  Below, we give the proof in full detail as it gives credence to the claim that the lower bound in~\eqref{eqn:lowb1} is sharp.    

\begin{Theorem}
Suppose that \cref{assump:1} is satisfied.
For all positive $\epsilon>0$, 
\[\TV{P_t(0,\cdot)-\pi_n} \geq 1-\epsilon\] 
provided that $n$ is large enough and \[t\leq \frac{\log n}{2|\log(1-2\lambda)|}-\frac{\log(\sqrt{3} +100)- \tfrac{1}{2}\log \epsilon }{ |\log (1-2\lambda)|}.\]
\end{Theorem}

\begin{proof}
Let $\pi$ be the stationary distribution of $(X_t)$ with $n$ suppressed.  
Let $f(x) = 
\sqrt{n-1}(1-\frac{2x}{n})$.  Since $\pi$ has hypergeometric distribution as in~\eqref{eqn:pidist}:
\[\E_{\pi} f(X_t) = \sqrt{n-1}\, \E_{\pi} f_1(X_t)=0.\]
Also note that 
\begin{align*}
    \text{Var}_{\pi} f(X_t) &= \E_{\pi}f(X_t)^2-(\E_{\pi} f(X_t))^2\\
    &= (n-1) \E_{\pi} f_1(X_t)^2\\
    &= \frac{n-1}{2n-1} +(n-1)\frac{2n-2}{2n-1} f_2(k)^t\E_{\pi}f_2(X_t)
\end{align*}
where the last equality follows from \cref{lem:1}.
Since $f_2(x)$ is an (right) eigenfunction of the chain with eigenvalue $f_2(k)\neq 1$, it is orthogonal to $\pi$.
In particular, $\E_{\pi} f_2(X_t)=0$. 
Thus, $\text{Var}_{\pi} f(X_t)\leq 1/2$. 

We repeat these calculations for the chain started from $0$ for $t$ satisfying 
\begin{align}
\label{eqn:tdef2}
t= \frac{\log n}{2|\log(1-2\lambda)|}-c=t_n-c 
\end{align}
for a constant $c$ which we will determine later.  Observe by \cref{lem:1} and \cref{lem:2} 
\begin{align}\E_{0} f(X_t) = \sqrt{n-1}\E_0 f_1(X_t) &=\sqrt{n-1} (1-\tfrac{2k}{n})^t \nonumber\\ \label{eqn:310}
&=\sqrt{n-1} (1-2\lambda)^t(1+o(1))\\
&= \frac{\sqrt{n-1}}{\sqrt{n}} (1-2\lambda)^{-c}(1+o(1))\nonumber\\ \label{eqn:311}
&= (1+o(1))(1-2\lambda)^{-c}.
\end{align}
Similarly, combining \eqref{eqn:310} with \cref{lem:1} and \cref{lem:2} we have 
\begin{align*}
    \text{Var}_{0} f(X_t)&= \E_0f(X_t)^2- (\E_0f(X_t))^2\\
    &= (n-1) \left[\frac{1}{2n-1}+ \frac{2n-2}{2n-1} f_2(k)^t f_2(0)\right] - (n-1)(1-2\lambda)^{2t}(1+o(1))\\
  &= (n-1) \left[\frac{1}{2n-1}+\frac{2n-2}{2n-1}(1-2\lambda)^{2t}( 1+ o(1))\right] - (n-1)(1-2\lambda)^{2t}(1+o(1))  \\
  &= \frac{n-1}{2n-1} + (n-1) (1-2\lambda)^{2t}(1+o(1))-(n-1)(1-2\lambda)^{2t}(1+o(1))\\
  &= \frac{1}{2}+ o(1) (n-1) (1-2\lambda)^{2t}= \frac{1}{2} +o(1) (1-2\lambda)^{-2c}.
          \end{align*}
In particular, $\text{Var}_{0}f(X_t) \leq 3/2$ for $n$ sufficiently large.

We finish the proof by applying Chebychev's inequality. 
Let 
\begin{align*}
A_\alpha = \{x\in \calx \, : \, |f(x)|\leq \alpha\} \qquad \text{ and } \qquad B_{r,c}=\{x\in \calx \, : \, |f(x) - \E_0 f(X_t)|\leq r\sqrt{3/2}\}
\end{align*}
where we recall that $t$ is as in~\eqref{eqn:tdef2}. 
Notice that \[\pi(A_\alpha)=\PP_\pi (\{|f(X_t)|\leq \alpha\})=1-\PP_\pi (\{|f(X_t)|>\alpha\})\geq 1- \frac{1}{2\alpha^2}\]
where we used $\PP_\pi(\{|f(X_t)|>\alpha\})\leq \alpha^{-2}\E_\pi |f(X_t)|^2 \leq 1/(2\alpha^2)$.  
On the other hand, in a similar way we have 
\[ {P_t}(0,B_{r,c}) = 1-P_t(0, B_{r,c}^{\text{c}}) =
1-\PP_0\{ |f(X_t)-\E_0f(X_t)|> r\sqrt{3/2}\}\geq 1-\frac{1}{r^2}\]
for $n$ sufficiently large.  
Notice that if $(1-2\lambda)^{-c}(1+o(1))-r\sqrt{3/2}\geq 100 \alpha$, then $A_\alpha$ and $B_{r,c}$ are disjoint.
Therefore, $\mu(B_
{r,c}^{\text{c}}) \geq \mu(A_\alpha)$ for any measure $\mu$ on subsets of $\calx$.
In particular, for appropriate choices of $\alpha, c, r$ we have that $P_t(0,A_\alpha)\leq 1/r^2$. 

Now let $\epsilon>0$ be small but fixed. Choose $\alpha, r$ such that $\frac{1}{2\alpha^2},\frac{1}{r^2}\leq \epsilon/2$ and pick $c$ such that $(1-2\lambda)^{-c}-r\sqrt{3/2}\geq 100\alpha$.  Note that under these constraints, we may choose $$c=\frac{\log(\sqrt{3} +100)- \tfrac{1}{2}\log \epsilon }{ |\log (1-2\lambda)|}.$$  
Then
\begin{align*}
    \TV{P_t(0,\,\cdot\,)-\pi (\, \cdot \,)}&\geq |P_t(0,A_\alpha)-\pi(A_\alpha)|= 1-\frac{1}{2\alpha^2}- \frac{1}{r^2}\geq 1-\epsilon. 
\end{align*}
This concludes the proof.
\end{proof}

\begin{Remark}
Notice that if $c\leq 0$, then $\E_{0}f(X_t)=(1+o(1))(1-2\lambda)^{-c}$ and will quickly approach $0$ as $c\to -\infty.$ 
In particular, $\E_{0}f(X_t)\to 0$ as $t$ increases beyond $t_n$. 
This roughly translates to the claim that $P_t(0,\cdot)$ and $\pi_n$ both assign almost all of their mass close to $n/2$. 
If this is true, then the total variation distance between $P_t(0,\cdot)$ and $\pi_n$ should be small.
This provides evidence for believing that the lower bound on the mixing time is sharp.
\end{Remark}
%
%

\section{Getting at distance $o(\sqrt{n})$}
\label{sec:osmallrootn}

Throughout this section, $(X_t)$ and $(Y_t)$ will denote two copies of the Bernoulli-Laplace chain.  The goal of this section is to show that in $3s_n +t_n$ steps (with $s_n, t_n$ as in~\eqref{def:times}) the chains are distance $o(\sqrt{n})$ apart with high probability.  The argument establishing this fact is broken into a few pieces based on the size of the initial distance $|X_0-Y_0|$.  First, we see that in $t_n$ steps, the two copies are within distance $O(\sqrt{n})$ (see~\cref{Lemma:Orootn} below) with high probability, regardless of their initial states.  Then, starting at distance $O(\sqrt{n})$ apart, we see that in \cref{Lemma:orootn} using the path coupling described in \cref{sec:outline}, the distance is decreased to $o(\sqrt{n})$ in $s_n$ steps with high probability.  Although the two processes become closer in this second step, there is a possibility that they slightly deviate from the mean  $n/2$ of the stationary distribution.  Consequently, in \cref{prop:returntomean} we show that in $2s_n$ steps, the chains return to the $O(\sqrt{n})$ band about $n/2$ while maintaining a distance of $o(\sqrt{n})$ between the two chains.        

\subsection*{Asymptotic notation}
Below, we adopt some further asymptotic notation for convenience of mathematical expression.  In what follows, if $a_n, b_n $ are sequences of nonnegative real numbers, we write $a_n \lesssim b_n$ if $a_n =O(b_n)$ as $n\rightarrow \infty$ and the constant in the asymptotic estimate does \emph{not} depend on the parameters $\kappa_i>0$ below.  We will also write $a_n \cong b_n$ if both $a_n \lesssim b_n$ and $b_n \lesssim a_n$.

\subsection{Getting at distance $O(\sqrt{n})$}

Applying \cref{lem:1} and~\cref{lem:2}, we will now show that running the chains $(X_t)$ and $(Y_t)$ for $t=t_n$ steps brings $X_t$ and $Y_t$ within distance $O(\sqrt n)$ of each other with high probability. 

\begin{Lemma}\label{Lemma:Orootn}
Suppose \cref{assump:1} is satisfied and   
for a fixed constant $\kappa_1\in (0,\infty)$, let \[\tau_1(\kappa_1):=\min \left\{t\,:\, X_t,Y_t\in (\tfrac{n}{2}-\kappa_1\sqrt{n},\tfrac{n}{2}+\kappa_1\sqrt{n}) \right\}.\]
Then, \[\PP(\tau_1(\kappa_1)>t_n) \lesssim  \frac{1}{\kappa_1^2}.\]
 \end{Lemma}
\begin{proof}
By combining Chebychev's inequality with \cref{lem:1} and~\cref{lem:2} we have
\begin{align}
\nonumber \PP_{X_0}\left(\left|X_{t_n}-\tfrac{n}{2}\right|>\kappa_1 \sqrt{n}\right)= \PP(f_1(X_{t_n})^2>4\kappa_1^2/n)&\leq \frac{n}{4 \kappa_1^2} \E_{X_0} f_1(X_{t_n})^2\\
\nonumber &= \frac{n}{4\kappa_1^2}[\tfrac{1}{2n-1} + \tfrac{2n-2}{2n-1}f_2(k)^{t_n} f_2(X_0)]\\
&\lesssim \frac{1}{\kappa_1^2} + \frac{n}{\kappa_1^2}(1-2\lambda)^{2t_n} \lesssim \frac{1}{\kappa_1^2} \label{cheby3}  
\end{align}
where in the asymptotic inequality~\eqref{cheby3} we used~\eqref{eqn:f2bound} to bound $f_2(X_0)$.  
Thus, 
\[\PP\left(\tau_1(\kappa_1)> t_n\right)\leq \PP_{X_0} \left(|X_{t_n}-\tfrac{n}{2}|>\kappa_1\sqrt{n}\right)+\PP_{Y_0} \left(|Y_{t_n}-\tfrac{n}{2}|>\kappa_1\sqrt{n}\right)\lesssim \frac{1}{\kappa_1^2}\]
by the estimate \eqref{cheby3}. 
\end{proof}

Given that the chains are sufficiently close to $n/2$, we want to ensure that they stay near $n/2$ for a significant window of time.  \cref{prop:step1} below shows that with high probability, the chains will stay close to the mean for $s_n=\lambda^{-1}\log\log n$ steps after falling within the $\kappa_1\sqrt{n}$ range of $n/2$.  Later, this provides ample time to move the chains closer without deviating too far from the mean of the stationary distribution.

\begin{Proposition}\label{prop:step1}
Suppose that~\cref{assump:1} is satisfied and that $|X_{0}-\tfrac{n}{2}|<\kappa_1 \sqrt{n}$. Then, for every $r\in (0, \infty)$, $s\in\N$ we have 
\[\PP\Big(\sup\limits_{t\in I_n(s)} |X_t-\tfrac{n}{2}|>r\Big)\lesssim \frac{n}{r^2}(\log n)^{|P_\lambda|},\]
where $I_n(s)=[s, s+ s_n]$, $P_\lambda = \log(1-2\lambda)\frac{2}{\lambda}$ and $s_n= \lambda^{-1}\log \log n$ is as in~\eqref{def:times}. 
\end{Proposition}

\begin{proof}Let $s \in \mathbf{N}$.  For $t\ge 0$, note that $|X_t-\tfrac{n}{2}|>r$ if and only if $f_1(X_t)=|1-2X_t/n|>2r/n$.
Define $M_t =f_1(X_t)/f_1(k)^t$. It follows that $M_t$ is a martingale by~\cref{lem:1} relation \eqref{id:1}.  Setting $s_n=  \lambda^{-1} \log \log n$, we thus obtain
\begin{align*}
    \PP\Big(\sup\limits_{t\in I_n(s)} \left|X_t-\tfrac{n}{2}\right|>r\Big)=\PP\Big(\sup\limits_{t\in I_n(s)} |f_1(X_t)|>\tfrac{2r}{n}\Big)&=\PP\Big(\sup\limits_{t\in I_n(s)}f_1(k)^t |M_t|>\tfrac{2r}{n}\Big)\\
    &\leq \PP\bigg(\sup\limits_{t\in I_n(s)} |M_t|>\frac{2r}{n}\frac{1}{f_1(k)^s }\bigg)\\
    &\leq \frac{n^2f_1(k)^{2s}}{4r^2}\E|M_{s+s_n} |^2,
\end{align*}
where on the last inequality we used the fact that $|M_t|^2$ is a submartingale.  Note that\begin{align*}
  \frac{n^2f_1(k)^{2s}}{4r^2}\E|M_{s+s_n} |^2
    &=\frac{n^2}{4r^2}\frac{1}{f_1(k)^{2s_n}}\E f_1(X_{s+s_n})^2\\
    &=\frac{n^2}{4r^2}\frac{1}{\left(1-\frac{2k}{n}\right)^{2s_n}}\left(\frac{1}{2n-1}+\frac{2n-2}{2n-1}f_2(k)^{s+s_n}f_2(X_0)\right).  
\end{align*}
where on the last line we use \cref{lem:1} relation~\eqref{id:2}.  

Now observe  that by \cref{lem:1}, $f_2(X_0)=O(n^{-1})$ since $|X_0- \tfrac{n}{2}| \leq \kappa_1 \sqrt{n}$.  Also, using~\cref{lem:2} with \cref{assump:1} implies $f_2(k)^{s+s_n}= o(1)$ as
 $n\rightarrow \infty$.  Combining the previous two observations and using \cref{lem:2} again produces
\begin{align*}
    \frac{n^2}{4r^2}\frac{1}{\left(1-\frac{2k}{n}\right)^{2s_n}}\left(\frac{1}{2n-1}+\frac{2n-2}{2n-1}f_2(k)^{s+s_n}f_2(X_0)\right)
    \cong \frac{n}{8r^2}\frac{1}{\left(1-\frac{2k}{n}\right)^{2s_n}} \cong \frac{n}{8r^2}(\log n)^{|P_\lambda|}.    
\end{align*}
Thus,
\[\PP\bigg(\sup\limits_{t\in I_n(s)} \left|X_t-\tfrac{n}{2}\right|>r\bigg)\lesssim \frac{n}{r^2}(\log n)^{|P_\lambda|} ,\] as claimed.  
\end{proof}

\subsection{Getting at distance $o(\sqrt{n})$ from $O(\sqrt{n})$}

Now, given that $X_0$ and $Y_0$ are within distance $O(\sqrt{n})$ of $n/2$, \cref{Lemma:orootn} below will bring the two chains $X_t$ and $Y_t$ within distance $o(\sqrt{n})$ of each other with high probability in a negligible number of steps.  This result is a modified version of~\cite[Lemma 12]{EN}.  Note, however, that although the distance between the chains decreases, the chains may slightly deviate from the mean.  See the choice of $r_n=\sqrt{n} (\log n)^{|P_\lambda|/2}$ below in the statement of~\cref{Lemma:orootn}.

\begin{Lemma}\label{Lemma:orootn}
Suppose that~\cref{assump:1} is satisfied and that $X_0,Y_0\in (\frac{n}{2}-\kappa_2\sqrt{n},\frac{n}{2}+\kappa_2\sqrt{n})$ for some $\kappa_2\in (0,\infty)$.  Suppose, furthermore, that the chains $(X_t)$,$(Y_t)$ are coupled according \cref{coupling}.  For $\kappa_3\in (0,\infty)$, consider the stopping time 
\[\tau_3(\kappa_3):= \min \left\{ t:|X_t-Y_t|\leq \frac{\sqrt{n}}{\log \log n}\quad\text{ and }\quad X_t,Y_t\in \left(\tfrac{n}{2}-\kappa_3r_n,\tfrac{n}{2}+\kappa_3r_n\right )\right\}\]
where $r_n =\sqrt{n}(\log n)^{|P_\lambda|/2}.$
Then we have 
\[\PP\left(\tau_3(\kappa_3)>s_n\right)\lesssim \frac{1}{\kappa_3^2}\]
where $s_n= \lambda^{-1} \log \log n$ is as in~\eqref{def:times}. 
\end{Lemma}


\begin{proof}
Recall by~\cref{assump:1}, $k<\tfrac{n}{2}$.   Also recall that (see~\eqref{eqn:tcouple})
\[\tau_{\emph{couple}}\left(\frac{\sqrt{n}}{\log \log n}\right)=\min\left\{t: |X_t-Y_t|\leq \frac{\sqrt{n}}{\log\log n}\right\}.\] 
By definition of $\tau_3(\kappa_3)$,
\begin{align*}
\left\{\tau_3(\kappa_3)>s_n\right\}&\subseteq\Big\{\tau_{\emph{couple}}\left(\tfrac{\sqrt n}{\log\log n}\right)> s_n\Big\} \cup \bigg( \bigcup_{t\in [0,s_n]}\left\{|X_t-\tfrac{n}{2}|\vee|Y_t-\tfrac{n}{2}|>\kappa_3 r_n \right\}\bigg)  .\end{align*}
By \cref{prop:couple}, the assumption on $|X_0-Y_0|$ and \cref{lem:2}, we know that 
\begin{align*}
    \PP\Big(\tau_{\emph{couple}}\left(\tfrac{\sqrt{n}}{\log\log n}\right)>s_n\Big)
    &\lesssim (1-2\lambda +2\lambda^2)^{s_n}\frac{2\kappa_2\sqrt{n}}{\tfrac{\sqrt{n}}{\log\log n}}\\
    &= e^{\log(1-2\lambda + 2\lambda^2)\lambda^{-1}\log\log n}2 \kappa_2\log\log n\\
    &=\frac{2\kappa_2\log\log n}{(\log n)^{|\log(1-2\lambda+2\lambda^2)|\lambda^{-1}}} =o(1).  
\end{align*}
Applying \cref{prop:step1} we get 
\[\PP\bigg(\bigcup_{t\in [0,s_n]}\left\{|X_t-\tfrac{n}{2}|\vee|Y_t-\tfrac{n}{2}|>\kappa_3r_n )\right\}\bigg)\lesssim \frac{n}{\kappa_3^2 r_n^2}(\log n)^{|P_\lambda|}\lesssim \frac{1}{\kappa_3^2}.\]
Using a union bound, we obtain the claimed result.  
\end{proof}

Recall that although \cref{Lemma:orootn} brings $X_t$ and $Y_t$ within distance $o(\sqrt{n})$ of one another in a negligible amount of time,  it is possible for $X_t$ and $Y_t$ to deviate from $n/2$ by more than $O(\sqrt{n})$. 
The idea in \cref{prop:returntomean} is that, even if this does happen, by taking another negligible number of steps we can ensure that with high probability the copies will return to the desired distance apart within the desired  $O(\sqrt{n})$ window about $n/2$.   


\begin{Proposition}\label{prop:returntomean}
Suppose \cref{assump:1} is satisfied and that $(X_t)$ and $(Y_t)$ are coupled as in \cref{coupling} with 
\begin{align*}
X_0,Y_0\in \left(\tfrac{n}{2}-\kappa_3\sqrt{n}(\log n)^{\frac{|P_{\lambda}|}{2}},\tfrac{n}{2}+\kappa_3\sqrt{n}(\log n)^{\frac{|P_{\lambda}|}{2}}\right) \qquad \text{ and } \qquad |X_0-Y_0| \leq \frac{\sqrt{n}}{\log\log n}
\end{align*}
for some $\kappa_3\in (0,\infty)$. 
For $\kappa_4\in (0,\infty)$, define the stopping time 
\[\tau_4(\kappa_4):= \min \left\{ t:|X_t-Y_t|\leq \frac{\sqrt{n}}{\log \log n}\,\,\,\text{ and }\,\,\, X_t,Y_t\in \left(\tfrac{n}{2}-\kappa_4\sqrt{n},\tfrac{n}{2}+\kappa_4\sqrt{n}\right )\right\}.\]
Then we have 
\[\PP\left(\tau_4(\kappa_4)>2s_n \right)\lesssim \frac{1}{\kappa_4^2}\]
where $s_n$ is as in~\eqref{def:times}.  
\end{Proposition}

\begin{proof}
By construction of the \cref{coupling}, $t\mapsto |X_t -Y_t|$ is monotone decreasing.  Hence, by the hypothesis on the initial condition we have \[|X_t-Y_t|\leq |X_0-Y_0| \leq \frac{\sqrt{n}}{\log \log n}\]
for all $t\geq 0$.  
In particular, \[\tau_4(\kappa_4)= \min \left\{ t: X_t,Y_t\in \left(\frac{n}{2}-\kappa_4\sqrt{n},\frac{n}{2}+\kappa_4\sqrt{n}\right )\right\}.\]
Using \cref{lem:1} with~\cref{lem:2} gives
\begin{align}
   \nonumber \E_{X_0} f_1(X_{2s_n})^2 &= \frac{1}{2n-1} + \frac{2n-2}{2n-1} f_2(k)^{2s_n} f_2(X_0)\\
  \label{eqn:cong1} & \cong \frac{1}{2n-1} + \frac{2n-2}{2n-1} (1-2\lambda)^{4s_n} f_2(X_0).  
    \end{align}
    By hypothesis on $X_0$, it follows by~\cref{lem:2} that
    \begin{align*}
    f_2(X_0) = O\Big( \frac{(\log n)^{|P_\lambda|}}{n}\Big).  
    \end{align*}
    Plugging this into~\eqref{eqn:cong1} implies 
    \begin{align*}
   \E_{X_0} f_1(X_{2s_n})^2 &\lesssim \frac{1}{2n-1} + \frac{2n-2}{2n-1} O(n^{-1} (\log n)^{- |P_\lambda|})\\
   & = O(n^{-1})  
   \end{align*}   
with the same estimate holding true for $Y_{2s_n}$.  

%

Combining these estimates with the monotonicity of the coupling along with a union bound and Chebychev's inequality then gives 
\begin{align*}
\PP_{X_0,Y_0}(\tau_4(\kappa_4)>2 s_n)&\leq \PP_{X_0,Y_0}\left(|X_{2s_n}-\tfrac{n}{2}|\vee |Y_{2s_n}-\tfrac{n}{2}|> \kappa_4\sqrt{n}\right)\\
&=\PP_{X_0,Y_0}\left(|f_1(X_{2s_n})|\vee |f_2(Y_{2s_n}) |> 2\kappa_4/\sqrt{n}\right) \lesssim  \frac{1}{\kappa_4^2},  
\end{align*}
finishing the proof.  
\end{proof}

\section{The upper bound on the mixing time}
\label{sec:upperbound}

If $X$ is a random variable, then we let $\mu_X$ denote its distribution.  The goal of this section is to show that if $X_0-Y_0=o(\sqrt{n})$ and $X_0=\tfrac{n}{2}+O(\sqrt{n})$, $Y_0=\tfrac{n}{2}+O(\sqrt{n})$ as $n\rightarrow \infty$, then 
\begin{align}
\label{eqn:TVo1}
\TV{P_1(X_0, \, \cdot \,) - P_1(Y_0, \, \cdot \,) }=\TV{\mu_{X_1}-\mu_{Y_1}} = o(1) \,\, \text{ as }\,\, n\rightarrow \infty.  
\end{align}
Combining this with the results of the previous section ultimately yields the desired upper bound on the mixing time.    

Our method to show~\eqref{eqn:TVo1} is quite different than the method used in~\cite{EN}.  To see why,   
let $\text{Hyper}(j, \ell, m)$ denote the hypergeometric distribution of $m$ objects selected without replacement from a total of $j$ objects, $\ell$ of which are of type 1 and $j-\ell$ are of type 2.  Also, let $\text{Bin}(j, r)$ denote the binomial distribution of $j$ trials with success probability $r$.  Observe that we can write 
\begin{align}
X_1-X_0= H_1 - H_0 \qquad \text{ and } \qquad Y_1 - Y_0 = H_3 - H_2 
\end{align}
where $H_1\sim \text{Hyper}(n,n-X_0, k)$ and $H_0\sim \text{Hyper}(n, X_0, k)$ while $H_3 \sim \text{Hyper}(n,n-Y_0, k)$ and $H_2\sim \text{Hyper}(n, Y_0, k)$.  Thus we wish to estimate
\begin{align}\label{eqn:hyperrep}
\TV{\mu_{X_1}-\mu_{Y_1}}= \TV{ \mu_{X_0+ H_1 - H_0} -\mu_{Y_0 +H_3- H_2} }.  
\end{align} 

In order to estimate~\eqref{eqn:hyperrep}, we employ a large $n$ comparison of the hypergeometric distributions $H_i$, $i=0,1,2,3$, in the total variation distance.  Because $k/n\rightarrow \lambda \in (0,1/2)$ instead of $k=o(n)$ as $n\rightarrow \infty$, we cannot appeal to the result of Diaconis and Freedman~\cite{DF_80}, which in particular shows that if $k=o(n)$, then 
\begin{align}
\label{eqn:hyperb}
\TV{\mu_{H_i} -\mu_B}\lesssim \frac{k}{n} + \sqrt{\frac{k}{n}}
\end{align}
where $B\sim \text{Bin}(k, 1/2)$.  Using normal approximation of the binomial, the desired result~\eqref{eqn:TVo1} follows~\cite[Lemma 16]{EN} in the case when $k=o(n)$ by using the triangle inequality several times in~\eqref{eqn:hyperrep}.       

Note that in our case, the righthand side of~\eqref{eqn:hyperb} is order $1$.  Hence, instead of using a binomial and then normal approximation, we manufacture a discrete version of the normal that approximates the $H_i$'s in total variation.  Combining this approximation with the triangle inequality a few times in~\eqref{eqn:hyperrep}, we will then be able to conclude~\eqref{eqn:TVo1}.        

\subsection{The discrete normal distribution}
In what follows, we will use $\phi :\mathbf{R}\rightarrow \mathbf{R}$ to denote the probability density function of the standard normal on $\mathbf{R}$; that is,
\begin{align}
\phi(x) = \frac{e^{-\frac{x^2}{2}}}{\sqrt{2\pi}}, \qquad x\in \mathbf{R}. 
\end{align}

\begin{Definition}
We say that a random variable $Z$  distributed on the integers $\mathbf{Z}$ has a \emph{discrete normal distribution} with parameters $x\in \mathbf{R}$, $s>0$ and finite set $\cals  \subset \mathbf{Z}$, denoted by $Z \sim \text{dN}(x, s, \mathcal{S})$, if 
\begin{align*}
\PP(Z= j)  = 
\begin{cases}
\frac{1}{\mathcal{N}}\frac{\phi \big( \frac{j-x}{s}\big)}{s} & \text{ if } \,\, j \in \cals\\
0 & \text{ if } j \notin \cals
\end{cases}
\end{align*}
where $\mathcal{N}:= \sum_{j\in \cals} \frac{\phi \big( \frac{j-x}{s}\big)}{s}$ is the normalization constant.  \end{Definition}

\begin{Remark}
Depending on the choice of parameters $x\in \mathbf{R}$, $s>0$  and finite set $\cals$, the name \emph{discrete normal} above is a bit of a misnomer.  Indeed, unless properly shifted and scaled with a large enough set $\cals$, a random variable $Z\sim \text{dN}(x,s, \cals)$ may have little to do with the usual normal distribution.  All of the discrete normals used below, however, will indeed be reminiscent of a normal distribution due to the particular choices of $x,s$ and $\cals$.     
\end{Remark}

In order to setup and prove the next result, set $\calx_k= \{0,1,2,\ldots, k\}$, fix $\ell \in \calx$ and define parameters
\begin{align}
\label{eqn:not1}
p:=\frac{\ell}{n},\, \,\, f:= \frac{k}{n},\, \,\,q:=1-p, \,\,\, \sigma:= \sqrt{k pq (1-f)}.
\end{align}
Furthermore, for $j\in \calx_k$ set
\begin{align}
\label{eqn:not2}
 x_j := \frac{j - kp}{\sqrt{k p q}} \qquad \text{ and } \qquad \tilde{x}_j := \frac{x_j}{\sqrt{1-f}}= \frac{j-kp}{\sigma}.  
\end{align}
We first need a technical lemma concerning the behavior of the normalization constant for a particular discrete normal. 
\begin{Lemma}
\label{lem:normalization}
Given the choice of parameters in~\eqref{eqn:not1} and~\eqref{eqn:not2}, suppose \cref{assump:1} is satisfied and that $\ell=n/2+O(\sqrt{n})$.  Then $$\mathcal{N}_n:=\sum_{j\in \calx_k}\frac{\phi(\tilde{x}_j)}{\sigma}=1 +O(n^{-1/2})\,\,\, \text{ as }\,\,\, n\rightarrow \infty.$$ 
\end{Lemma}

\begin{proof}
The proof of this result follows by integral comparison. Note that 
\begin{align*}
\mathcal{N}_n &\leq \frac{1}{\sqrt{2\pi \sigma^2}}+ \sum_{j=0}^{\lfloor kp\rfloor-1}\frac{ e^{-\frac{1}{2}\big(\frac{j-kp}{\sigma} \big)^2}}{\sqrt{2\pi \sigma^2}} + \sum_{j=\lfloor kp\rfloor +1}^k \frac{ e^{-\frac{1}{2}\big(\frac{j-kp}{\sigma} \big)^2}}{\sqrt{2\pi \sigma^2}}\\
& \leq \int_\mathbf{R} \frac{ e^{-\frac{1}{2}\big(\frac{x-kp}{\sigma} \big)^2}}{\sqrt{2\pi \sigma^2}} \, dx + \frac{1}{\sqrt{2\pi \sigma^2}}=1 +   \frac{1}{\sqrt{2\pi \sigma^2}}.   
\end{align*} 
Given the asymptotic behavior of $\sigma$, this finishes the proof of the upper bound.  To obtain the lower bound, note that 
\begin{align*}
\mathcal{N}_n &\geq  \sum_{j=1}^{\lfloor kp\rfloor }\frac{ e^{-\frac{1}{2}\big(\frac{j-kp}{\sigma} \big)^2}}{\sqrt{2\pi \sigma^2}} + \sum_{j=\lfloor kp\rfloor}^k \frac{ e^{-\frac{1}{2}\big(\frac{j-kp}{\sigma} \big)^2}}{\sqrt{2\pi \sigma^2}} - \frac{1}{\sqrt{2\pi \sigma^2}} \\
& \geq \int_0^k \frac{ e^{-\frac{1}{2}\big(\frac{x-kp}{\sigma} \big)^2}}{\sqrt{2\pi \sigma^2}} \, dx  - \frac{1}{\sqrt{2\pi \sigma^2}}\\
&= 1- \int_{-\infty}^{-kp/\sigma} \phi(x) \, dx - \int_{kq/\sigma}^\infty \phi(x) \, dx - \frac{1}{\sqrt{2\pi \sigma^2}}. \end{align*}
Now if $\mathcal{Z}$ is a standard normal random variable on $\mathbf{R}$, we note that 
\begin{align*}
\mathcal{N}_n &\geq 1- \int_{-\infty}^{-kp/\sigma} \phi(x) \, dx - \int_{kq/\sigma}^\infty \phi(x) \, dx - \frac{1}{\sqrt{2\pi \sigma^2}} \\
&=1- \PP(\mathcal{Z}\leq -k p/\sigma) -\PP(\mathcal{Z} \geq kq/\sigma)- \frac{1}{\sqrt{2\pi \sigma^2}}. 
\end{align*}
By Chebychev's inequality on the square $\mathcal{Z}^2$, we see that both $\PP(\mathcal{Z}\leq -kp/\sigma)$ and $\PP(\mathcal{Z} \geq kq/\sigma)$ are order $1/n$ as $\ell=n/2 + O(\sqrt{n})$. We thus conclude the result.  
\end{proof}

Before stating the main comparison lemma, we will use a result from \cite{LCM_07}.
Though the following lemma is not stated explicitly in \cite{LCM_07}, the proof of it is contained within the proof of \cite[Theorem~1]{LCM_07}.
In particular, we will restate the assumptions that are required for \cite[Equations~(4.28)-(4.29)]{LCM_07}.

\begin{Lemma}\label{StatLemma}
Suppose $0<f<1$, $0<p<1$, and $6(kp\wedge kq)\geq 1$ where $f= k/n,$ $p= \ell / n$, and $q= 1-p$. 
Let \[L = \inf\{j\in \Z_+ :\tilde x_j\geq -\delta \sigma\} \quad \text{and}\quad J_x = \lfloor kp - x\sigma \rfloor\] for $x\in \RR$. 
For any $\delta \in (0,1/2]$ and $x\in [-\delta \sigma, 0]$, we have 
\[ \sum_{j=L}^{J_x}|\PP(H=j)-\mu(j)|\leq \frac{C}{\sigma (1-f)}[(1+x^2)\exp(-0.07x^2).\]
\end{Lemma}

We will use the previous lemma in the next result, which provides the critical comparison.  

\begin{Proposition}\label{hyper}
Given the choice of parameters in~\eqref{eqn:not1} and~\eqref{eqn:not2}, suppose \cref{assump:1} is satisfied.  Let $H\sim \text{\emph{Hyper}}(n, \ell,k)$ and $Z \sim \text{\emph{dN}}(kp, \sigma, \calx_k)$ with $\ell= n/2+O(\sqrt{n})$.  Then 
\begin{align}
\TV{\mu_H-\mu_Z}= O(n^{-1/2}) \,\,\text{ as } \,\, n \rightarrow \infty. 
\end{align}    
\end{Proposition}

\begin{proof}
Using the notation above in~\eqref{eqn:not1} and~\eqref{eqn:not2}, we first introduce some further notation that helps connect with the setup in~\cite{LCM_07}.  Let  
\begin{align*}
\bar{f}= f\wedge (1-f), \,\,  a:= \frac{\bar{f} +4}{4(1-\bar{f})}, \,\, \, \delta := \frac{1}{10(a\vee 2)} 
\end{align*}
and define
\[L=\inf\{ j \in \calx_k \,: \,\tilde x_j\geq -\delta\sigma\} \qquad\text{and} \qquad R=\sup\{j \in \calx_k \,:\, \tilde x_j\leq \delta\sigma\}.\]
Setting $\mu(j)= \PP(Z=j)$, we find that 
\begin{align*}
    2\TV{\mu_H-\mu_Z}&=\sum_{j=0}^{L-1}|\PP(H=j)-\mu(j))|+\sum_{j=L}^{R}|\PP(H=j)-\mu(j)|+\sum_{\ell=R+1}^{k}|\PP(H=j)- \mu(j)|\\
    & \leq \sum_{j=0}^{L-1}(\PP(H=j)+\mu(j))+\sum_{j=L}^{R}|\PP(H=j)-\mu(j)|+\sum_{\ell=R+1}^{k}(\PP(H=j)+\mu(j))\\
    & \leq \sum_{j=L}^{R}|\PP(H=j)-\mu(j)|+ \frac{2k e^{-\delta^2 \sigma^2}}{\mathcal{N}_n\sqrt{2\pi\sigma}} + \PP(H < L) + \PP(H > R).
    \end{align*}
    Using the well known tail bound in \cite{Hoeffding_1963}, we arrive at 
    \begin{align*}
    P(H<L) + P(H>R) \leq 2e^{-2k\left(\delta pq(1-f)\right)^2}\lesssim \frac{1}{\sqrt{n}}. 
    \end{align*}
    Furthermore, it follows from symmetry, \cref{lem:normalization}, and \cref{StatLemma} that
    \begin{align*}
     \sum_{j=L}^{R}|\PP(H=j)-\mu(j)| \leq2\sum_{j=L}^{\lfloor kp\rfloor}\bigg|\PP(H=j)-\frac{\phi((j-kp)/\sigma)}{\sigma}\bigg| + O(n^{-1/2})\leq \frac{C}{\sigma} \lesssim \frac{1}{\sqrt{n}}.  
    \end{align*}
    This concludes the proof.
\end{proof}

\subsection{Using the hypergeometric comparison}
Let us now combine the previous result with~\eqref{eqn:hyperrep} in order to see what we have left to estimate.  To this end, let $Z_i \sim \text{dN}(k p_i, \sigma_i, \calx_k), i=0,1,2,3$, with
\begin{align}
\label{eqn:not3}
\ell_0= X_0, \,\,\, \ell_1=n-X_0, \,\,\, \ell_2= Y_0, \,\,\, \ell_3=n-Y_0, \,\,\, p_i=\ell_i/n,
\end{align}  
and 
\begin{align}
\label{eqn:not4}
\sigma_i = \sqrt{kp_i (1-p_i) (1-k/n)}, i=0,1,2,3. 
\end{align}
Note that if $\eta= X_0-Y_0$, then by the triangle inequality and properties of the total variation distance of random variables distributed on $\mathbf{Z}$
\begin{align*}
\TV{\mu_{X_1}- \mu_{Y_1}} &= \TV{\mu_{X_0+ H_1-H_0}- \mu_{Y_0+ H_3-H_2} }\\
    & \leq \TV{\mu_{X_0 + H_1-H_0} - \mu_{X_0+ Z_1-Z_0} } + \TV{\mu_{X_0 +Z_1 -Z_0}- \mu_{Y_0 +Z_3-Z_2}}\\
    & \qquad + \TV{\mu_{Y_0 +Z_3 -Z_2}- \mu_{Y_0 +H_3-H_2}}\\
    &= \TV{\mu_{H_1-H_0} - \mu_{Z_1-Z_0} } + \TV{\mu_{\eta+Z_1 -Z_0}- \mu_{Z_3-Z_2}}\\
    & \qquad + \TV{\mu_{Z_3 -Z_2}- \mu_{H_3-H_2}}\\
    & \leq \sum_{i=0}^3\TV{\mu_{H_i}- \mu_{Z_i}} + \| \mu_{\eta + Z_1}- \mu_{Z_3} \|_{TV} + \|\mu_{Z_0}- \mu_{Z_2} \|_{TV}.  
    \end{align*}
    By \cref{hyper}, 
    \begin{align*}
    \sum_{i=0}^3\TV{\mu_{H_i}- \mu_{Z_i}} = O(n^{-1/2}) 
        \end{align*}
   assuming $\eta=X_0-Y_0=o(\sqrt{n})$ and $X_0= \tfrac{n}{2}+O(\sqrt{n}), Y_0= \tfrac{n}{2}+O(\sqrt{n})$.     
Thus, we must bound 
\begin{align*}
\| \mu_{\eta + Z_1}- \mu_{Z_3} \|_{TV} + \|\mu_{Z_0}- \mu_{Z_2} \|_{TV}
\end{align*}
under the same assumptions.  
\begin{Lemma}\label{normal}
Suppose that \cref{assump:1} is satisfied and consider the parameters $\ell_i$, $p_i$, $\sigma_i$ as in \eqref{eqn:not3} and the random variables $Z_i \sim \text{\emph{dN}}(kp_i, \sigma_i, \calx_k)$.  If $\eta=X_0-Y_0=o(\sqrt{n})$ and $X_0=n/2+ O(n^{1/2})$, $Y_0= n/2+O(n^{1/2})$, then as $n\rightarrow \infty$ 
\begin{align}
\label{eqn:Zs}
\| \mu_{\eta + Z_1}- \mu_{Z_3} \|_{TV} + \|\mu_{Z_0}- \mu_{Z_2} \|_{TV} =o(1).
\end{align} 
\end{Lemma}

%
%
%
\begin{proof}
Note that it suffices to show that $\TV{\mu_{\eta+Z_1} - \mu_{Z_3}}\rightarrow 0$ as $n\rightarrow \infty$.   Let $\epsilon >0$, $\caly_\eta := \calx_k \cap (\calx_k + \eta)$ and $\mathcal{N}_{n,1}$ and $\mathcal{N}_{n,3}$ denote the respective normalization constants for $Z_1$ and $Z_3$.  Fixing a constant $K>0$ to be determined momentarily, let $J_n(K)=[\tfrac{k}{2}-K \sqrt{n}, \tfrac{k}{2}+ K \sqrt{n}]$.  Note we can write
\begin{align*}
2 \TV{\mu_{\eta+Z_1} - \mu_{Z_3}}&= \sum_{j\in \mathbf{Z}} | \PP(Z_1 = j- \eta) - \PP(Z_3=j) | \\
&= \sum_{j \in \caly_\eta\cap J_n(K)} \bigg|\frac{\phi\big(\frac{j-\eta-kp_1}{\sigma_1} \big)}{\mathcal{N}_{n,1}\sigma_1 }- \frac{\phi\big(\frac{j-kp_3}{\sigma_3} \big)}{\mathcal{N}_{n,3}\sigma_3 }\bigg|\\
&\qquad +  \sum_{j \in \caly_\eta\cap J_n(K)^c} \bigg|\frac{\phi\big(\frac{j-\eta-kp_1}{\sigma_1} \big)}{\mathcal{N}_{n,1}\sigma_1 }- \frac{\phi\big(\frac{j-kp_3}{\sigma_3} \big)}{\mathcal{N}_{n,3}\sigma_3 }\bigg|\\
&\qquad +  \sum_{j\in \calx_k - (\calx_k +\eta)} \PP(Z_3=j) +  \sum_{j\in (\calx_k +\eta)-\calx_k} \PP(Z_1=j)\\
&=: T_1 + T_2 + T_3 +T_4.    
\end{align*}     
We next show how to estimate $T_3$.  The term $T_4$ can be done analogously, so we omit those details.  Observe that if $\eta >0$ and $j\in \calx_k - (\calx_k + \eta)$, then $j \leq \eta-1$.  Also if $\eta \leq 0$ and $j \in \calx_k -(\calx_k + \eta)$, then $j \geq k+\eta+1$.  Now since $\eta=o(\sqrt{k})$
\begin{align*}
T_3 &\leq \sum_{j=0}^{|\eta|} \PP(Z_3=j) + \sum_{j=k+1-|\eta|}^k \PP(Z_3 = j)\\
& \lesssim \frac{(| \eta|+ 1)}{\sigma_3} \bigg(\phi\bigg( \frac{|\eta|-kp_3}{\sigma_3}\bigg)+ \phi\bigg( \frac{k+1 -|\eta| - kp_3}{\sigma_3}\bigg) \bigg) \lesssim e^{-\epsilon' n}
\end{align*}   
for some constant $ \epsilon' >0$ independent of $n$.    

We next estimate $T_2$.  Note that, by using integral comparison and that $\ell_i= n/2 + O(n^{1/2})$ and $\eta=o(n^{1/2})$, it follows that there exists a constant $C>0$ independent of $K,n$ such that for all $K$ and $n$ large enough
\begin{align*} 
T_2 \leq C \int_{-\infty}^{-K/2} \phi(x) \, dx + C \int_{K/2}^\infty \phi(x) \, dx.   
\end{align*}
Thus pick $K>0$ large enough so that $T_2<\epsilon/2$ for all $n$ large enough.  

Turning finally to $T_1$, first note that since $\eta=X_0-Y_0=o(n^{1/2})$, $X_0=n/2+ O(n^{1/2})$ and $Y_0= n/2+O(n^{1/2})$
\begin{align*}
\frac{1}{\sigma_1}- \frac{1}{\sigma_3}&= \frac{\eta + \frac{Y_0^2-X_0^2}{n}}{\sqrt{X_0 Y_0(1-X_0/n)(1-Y_0/n)(1-\frac{k}{n})}} \frac{\sqrt{n/k}}{\sqrt{X_0(1-X_0/n)}+ \sqrt{Y_0(1-Y_0/n)}}\\
&= o(n^{-1}).   
\end{align*}
Thus combining this with \cref{lem:normalization} produces 
\begin{align*}
T_1 &= \sum_{j \in \caly_\eta\cap J_n(K)} \bigg|\frac{\phi\big(\frac{j-\eta-k p_1}{\sigma_1} \big)}{\sigma_1 }- \frac{\phi\big(\frac{j-k p_3}{\sigma_3} \big)}{\sigma_3 }\bigg| + O(n^{-1/2})\\
&= \sum_{j \in \caly_\eta\cap J_n(K)} \bigg|\frac{\phi\big(\frac{j-\eta-kp_1}{\sigma_1} \big)}{\sigma_1 }- \frac{\phi\big(\frac{j-kp_3}{\sigma_3} \big)}{\sigma_1 }\bigg| + o(1)\\
&=: T_1' + o(1).\end{align*}
Note that for $j\in \calx_k \cap J_n(K)$, both $(j-\eta-kp_1)/\sigma_1$ and $(j-kp_3)/\sigma_3$ are bounded in $n$.  Hence for $T_1'$ we may write for some constant $C>0$ 
\begin{align*}
T_1' &=  \sum_{j \in \caly_\eta\cap J_n(K)} \bigg|\frac{\phi\big(\frac{j-\eta-kp_1}{\sigma_1} \big)}{\sigma_1 }- \frac{\phi\big(\frac{j-kp_3}{\sigma_3} \big)}{\sigma_1 }\bigg| \\
&= \frac{1}{\sqrt{2\pi \sigma_1^2}} \sum_{j\in \caly_\eta \cap J_n(K)} e^{-\frac{(j-\eta-k p_1)^2}{2\sigma_1^2}}\bigg| 1- e^{- \big(\frac{(j-kp_3)^2}{2\sigma_3^2}-\frac{(j-\eta-kp_1)^2}{2\sigma_1^2}\big) }\bigg| \\
& \leq  \frac{C}{\sqrt{2\pi \sigma_1^2}}  \sum_{j\in \caly_\eta \cap J_n(K)} e^{-\frac{(j-\eta-kp_1)^2}{2\sigma_1^2}}\bigg| \frac{j-kp_3}{\sigma_3}- \frac{j-\eta - kp_1}{\sigma_1}\bigg|\\
& =   \frac{C}{\sqrt{2\pi \sigma_1^2}}  \sum_{j\in \caly_\eta \cap J_n(K)} e^{-\frac{(j-\eta-k p_1)^2}{2\sigma_1^2}}\bigg| \frac{j-k p_3}{\sigma_1}- \frac{j-\eta - k p_1}{\sigma_1}\bigg| + O(n^{-1/2})\\
&\leq  \frac{C}{\sqrt{2\pi \sigma_1^2}}  \sum_{j\in \caly_\eta \cap J_n(K)} e^{-\frac{(j-\eta-k p_1)^2}{2\sigma_1^2}} \frac{2\eta}{\sigma_1} + O(n^{-1/2})= o(1)
\end{align*}
as $\eta= o(n^{1/2})$ and $\sigma= O(n^{1/2})$.  This completes the proof since by choosing $N=N(K)>0$ large enough we have 
\begin{align}   
\TV{\mu_{X_1} - \mu_{Y_1} } < \epsilon
\end{align}
for all $n\geq N$.  

\end{proof}

\subsection{Upper bound}

We now use the previous estimates to conclude the upper bound in~\eqref{eqn:upperb}.

\begin{Theorem}
Let $\epsilon >0$.  For the $(n,k)$-Bernoulli-Laplace model under \cref{assump:1}, we have 
\[t_{mix}(\epsilon)\leq \frac{\log(n)}{2|\log(1-2\lambda)|} + 3\lambda^{-1}\log\log n +1\] for large enough $n$. 
\end{Theorem}

In order to setup the proof, we recall the definitions of $\tau_1(\kappa_1), \tau_3(\kappa_3)$ and $\tau_4(\kappa_4)$:
\begin{align*}
\tau_1(\kappa_1)&:=\min \left\{t\,:\, X_t,Y_t\in (\tfrac{n}{2}-\kappa_1\sqrt{n},\tfrac{n}{2}+\kappa_1\sqrt{n}) \right\},\\
\tau_3(\kappa_3)&:= \min \left\{ t:|X_t-Y_t|\leq \frac{\sqrt{n}}{\log \log n}\quad\text{ and }\quad X_t,Y_t\in \left(\tfrac{n}{2}-\kappa_3r_n,\tfrac{n}{2}+\kappa_3r_n\right )\right\},\\
\tau_4(\kappa_4)&:= \min \left\{ t:|X_t-Y_t|\leq \frac{\sqrt{n}}{\log \log n}\,\,\,\text{ and }\,\,\, X_t,Y_t\in \left(\tfrac{n}{2}-\kappa_4\sqrt{n},\tfrac{n}{2}+\kappa_4\sqrt{n}\right )\right\},
\end{align*}
where $r_n =\sqrt{n}(\log n)^{|P_\lambda|/2}$.

\begin{proof}
Below, for simplicity of expression, we will suppress the $\kappa_i$'s in $\tau_i(\kappa_i)$.  For $t< \tau_4$, we couple the two chains $(X_t)$ and $(Y_t)$ according to \cref{coupling}.  For $t\geq \tau_4$, we pick the optimal coupling; that is, the coupling $(X_t, Y_t)$ of $\mu_{X_t}$ and $\mu_{Y_t}$ such that $\TV{\mu_{X_t}- \mu_{Y_t}} = \PP\{X_t \neq Y_t\}=\PP\{ |X_t - Y_t | \geq 1 \}$.    
Note by \cref{Lemma:Orootn} and the Strong Markov Property:  
\begin{align*}
    \PP(\tau_4 >t_n+3s_n \,|\,X_0,Y_0)&=\PP(\tau_4 >t_n+3s_n, \tau_1> t_n \,|\,X_0,Y_0)+\PP(\tau_4 >t_n+3s_n, \tau_1\leq  t_n \,|\,X_0,Y_0)\\
    &\lesssim \frac{1}{\kappa_1^2} + \PP(\tau_4 >t_n+3s_n, \tau_1\leq  t_n|X_0,Y_0)\\
    &= \frac{1}{\kappa_1^2} + \E\E(\textbf{1}\{\tau_4 >t_n+3s_n\}\textbf{1}\{ \tau_1\leq  t_n\}|\mathcal F_{\tau_1})\\
    &= \frac{1}{\kappa_1^2} + \E(\textbf{1}\{\tau_1\leq  t_n\}\PP (\tau_4 >t_n+3s_n|\mathcal F_{\tau_1}))\\
    &\leq \frac{1}{\kappa_1^2} + \E\PP (\tau_4 >3s_n\,|\, X_{\tau_1},Y_{\tau_1}).
\end{align*}
Continuing in this way, we obtain by \cref{Lemma:orootn}
\begin{align*}
     \E \PP (\tau_4 >3s_n\,|\,X_{\tau_1},Y_{\tau_1})&= \E \PP (\tau_4 >3s_n, \tau_3>s_n \,|\,X_{\tau_1},Y_{\tau_1})+ \E\PP (\tau_4 >3s_n,\tau_3\leq s_n\,|\,X_{\tau_1},Y_{\tau_1})\\
    &\lesssim \frac{1}{\kappa_3^2} + \E\PP (\tau_4 >3s_n,\tau_3\leq s_n\,|\,X_{\tau_1},Y_{\tau_1})\\
    &= \frac{1}{\kappa_3^2} + \E\E (\textbf{1}\{\tau_4 >3s_n\}\textbf{1}\{\tau_3\leq s_n\}\,|\,\mathcal{F}_{\tau_3})\\
    &= \frac{1}{\kappa_3^2} + \E(\textbf{1}\{\tau_3\leq s_n\}\PP(\tau_4> 3s_n| \mathcal F_{\tau_3}))\\
    &\leq \frac{1}{\kappa_3^2} + \E\PP(\tau_4>2s_n| X_{\tau_3}, Y_{\tau_3})\lesssim \frac{1}{\kappa_3^2}+\frac{1}{\kappa_4^3}.
\end{align*}
Combining the above with \cref{normal} and setting $t= t_n+ 3s_n + 1$ gives
\begin{align*}
\TV{\mu_{X_t}- \mu_{Y_t} }& \leq \PP ( |X_t - Y_t | \geq 1 ) \\
&= \PP ( |X_t - Y_t | \geq 1, \, \tau_4 > t_n + 3s_n ) + \PP ( |X_t - Y_t | \geq 1, \, \tau_4 \leq t_n + 3s_n )\\
& \lesssim \frac{1}{\kappa_1^2} + \frac{1}{\kappa_3^2}+ \frac{1}{\kappa_4^2} +  \PP ( |X_t - Y_t | \geq 1, \, \tau_4 \leq t_n + 3s_n )\\
&\leq \frac{1}{\kappa_1^2} + \frac{1}{\kappa_3^2}+ \frac{1}{\kappa_4^2} + \E \PP ( |X_1 - Y_1| \geq 1 \, | \, X_{\tau_4}, Y_{\tau_4} )\\
&=\frac{1}{\kappa_1^2} + \frac{1}{\kappa_3^2}+ \frac{1}{\kappa_4^2} + o(1). \end{align*}
Note that picking the $\kappa_i$'s and $n$ large enough finishes the proof. 
\end{proof}

\section{Proof of the auxiliary results}
\label{sec:aux}
In this section, we prove \cref{lem:1} and \cref{lem:2}.   

\begin{proof}[Proof of~\cref{lem:1}]
Let $(\mathcal{F}_t)_{t\geq 0}$ denote a filtration of $\sigma$-fields to which $(X_t)$ is adapted. We start with the proof of~\eqref{id:1} and~\eqref{id:2}.  These identities are contained in the proof of \cite[Lemma 3]{EN}, but we provide the details for completeness.  For~\eqref{id:1}, since $f_1(x)$ is an eigenfunction with eigenvalue $f_1(k)=1-2k/n$ (cf.~\cite{DS_87, EN}), we obtain for $t\geq 1$
\begin{align*}
\E_{X_0} f_1(X_t) = \E_{X_0} \E_{X_0}[ f_1(X_t) | \mathcal{F}_{t-1}] = \E_{X_0} \E_{X_{t-1}}[ f_1(X_1)] = (1-\tfrac{2k}{n}) \E_{X_0} f_1(X_{t-1}).  
\end{align*}   
Repeating the process above produces~\eqref{id:1}.  In order to obtain~\eqref{id:2}, by a direct calculation it follows that
\begin{align*}
f_1(x)^2= \frac{1}{2n-1} + \frac{2n-2}{2n-1} f_2(x).  
\end{align*}
Thus following the same eigenvalue procedure as above but $f_2$ playing the role of $f_1$ we obtain~\eqref{id:2}.  

In order to establish~\eqref{eqn:f2as}, let $q\geq 0$ and $X_0= \tfrac{n}{2}+ \zeta$ where $\zeta$ is $O(\sqrt{n} (\log n)^{q/2})$.  Then
\begin{align*}
f_2(X_0) &= \bigg(1+ \frac{2n-1}{2n-2} - \frac{2n-1}{n}\bigg) + \zeta\bigg(\frac{2(2n-1)}{n(n-1)}- \frac{2(2n-1)}{n^2} \bigg) + \zeta^2\frac{2(2n-1)}{n^2(n-1)} -  \frac{2n-1}{n(n-1)}\\
&\qquad \qquad - \zeta \frac{2(2n-1)}{n^2(n-1)}\\
&= \zeta^2\frac{2(2n-1)}{n^2(n-1)} +O(n^{-1}) = O(n^{-1} (\log n)^q),
\end{align*}  
as claimed.  
\end{proof}

\begin{proof}[Proof of~\cref{lem:2}]
We first prove~\eqref{as:id1}.  Letting $g: (-\infty, 1/2)\rightarrow \mathbf{R}$ be given by $g(x)= \log (1-2x)$, observe that 
\begin{align}
\label{eqn:as1}
f_1(k)^t = (1-2\lambda)^t \frac{(1-2k/n)^t}{(1-2\lambda)^t} &=(1-2\lambda)^t \exp ( t[g(k/n) - g(\lambda)])\\
\nonumber &=: (1-2\lambda)^t \exp( t \Delta(g)).  
\end{align} 
Taylor's formula then implies
\begin{align*}
\Delta(g)=g(k/n) -g(\lambda) = -\frac{2\Delta_n}{1-2\lambda} + \Delta_n^2 \frac{g''(\xi)}{2} 
\end{align*} 
for some $\xi $ in between $k/n$ and $\lambda$.  Employing \cref{assump:1} (c1) then gives
\begin{align*}
 \Delta_n^2\frac{g''(\xi)}{2!} = O( \Delta_n^2) .
 \end{align*} 
 Plugging this into~\eqref{eqn:as1} using $t\leq t_n$ and \cref{assump:1}(c2) then implies
 \begin{align*}
f_1(k)^t &= (1-2\lambda)^t \sum_{j=0}^\infty \frac{(t \Delta(g))^j}{j!}\\
 &= (1-2\lambda)^t(1- 2t \Delta_n/(1-2\lambda) + O(t^2 \Delta_n^2)),
 \end{align*}
 which is~\eqref{as:id1}.  
 
 In order to obtain~\eqref{as:id2}, we follow a similar reasoning to the one used to arrive at~\eqref{as:id1} to see that
 \begin{align*}
 f_2(k)^t = (1-2\lambda)^{2t} \frac{f_2(k)^t}{(1-2\lambda)^{2t}} = (1- 2\lambda)^{2t} \exp\{t (g(h_1(k))- g(h_2(\lambda )))\}. 
 \end{align*}   
A short calculation shows that
 \begin{align*}
 \Delta_n'= h_1(k)-h_2(\lambda) =2 \Delta_n (1-\tfrac{k}{n}- \lambda) +O(1/n),  
 \end{align*} 
 yielding the first part of~\eqref{eqn:dtil}.  Following the same reasoning as above yields~\eqref{as:id2}. 
 
 The proof of~\eqref{as:id3} and the second part of~\eqref{eqn:dtil} are similar, so we omit the details.  
 
\end{proof}

\bibliographystyle{plainurl}
\bibliography{bibliography}
%
%
%
%
%
%
%
%
%



\end{document}